\documentclass{amsart}

\usepackage{a4wide,amsmath,amssymb}
\usepackage[toc,page]{appendix}

\usepackage{url}

\DeclareMathAlphabet{\mathpzc}{OT1}{pzc}{m}{it}

\newtheorem{thm}{Theorem}[section]
\newtheorem{lemma}[thm]{Lemma}

\newtheorem{cor}[thm]{Corollary}
\theoremstyle{remark}
\newtheorem{rem}[thm]{Remark}

\newtheorem{cex}[thm]{Counterexample}
\theoremstyle{definition}
\newtheorem{defn}[thm]{Definition}
\newtheoremstyle{Claim}{}{}{\itshape}{}{\itshape\bfseries}{:}{ }{#1}
\theoremstyle{Claim}

\newcommand{\R}{\mathbb{R}}

\newcommand{\He}{\mathbb{H}}

\newcommand\sQ{\mathpzc{Q}}

\newcommand{\Mp}{\mathcal{M}^+_{\lambda,\Lambda}}
\newcommand{\Mm}{\mathcal{M}^-_{\lambda,\Lambda}}

\DeclareMathOperator{\LSC}{LSC}

\DeclareMathAlphabet{\mathpzc}{OT1}{pzc}{m}{it}

\theoremstyle{plain}

\def\sideremark#1{\ifvmode\leavevmode\fi\vadjust{
\vbox to0pt{\hbox to 0pt{\hskip\hsize\hskip1em
\vbox{\hsize3cm\tiny\raggedright\pretolerance10000
\noindent #1\hfill}\hss}\vbox to8pt{\vfil}\vss}}}

\begin{document}

\title{Some new Liouville-type results for fully nonlinear PDEs on the Heisenberg group}

\author{Alessandro Goffi} 

\date{\today}
 \thanks{
 This work has been partially supported by 
the Fondazione CaRiPaRo
Project ``Nonlinear Partial Differential Equations:
Asymptotic Problems and Mean-Field Games". The author is member of the Gruppo Nazionale per l'Analisi Matematica, la Probabilit\`a e le loro Applicazioni (GNAMPA) of the Istituto Nazionale di Alta Matematica (INdAM)
}
\address{Department of Mathematics ``T. Levi-Civita", University of Padova, Via Trieste 63, 35121 Padova, Italy} \email{alessandro.goffi@math.unipd.it}

\maketitle
\begin{abstract}
We prove new (sharp) Liouville-type properties via degenerate Hadamard three-sphere theorems for fully nonlinear equations structured over Heisenberg vector fields. As model examples, we cover the case of Pucci's extremal operators perturbed by suitable semilinear and gradient terms, extending to the Heisenberg setting known contributions valid in the Euclidean framework.
\end{abstract}
\tableofcontents

\section{Introduction}
\label{intro}
The purpose of this paper is to investigate Liouville properties for fully nonlinear degenerate differential inequalities of the form
\begin{equation}\label{1}
G(x,u,D_{\mathbb{H}^d}u,D^2_{\mathbb{H}^d}u)\leq 0\text{ in }\mathbb{H}^d\simeq \R^{2d+1}
\end{equation}
such as
\begin{itemize}
\item[(a)] any nonnegative solution to \eqref{1} is constant.
\item[(b)] the only nonnegative solution to \eqref{1} vanishes identically.
\end{itemize}
These properties have been widely investigated in the framework of viscosity solutions to fully nonlinear uniformly elliptic equations, see \cite{CIL,CC} for a comprehensive introduction. The Liouville property has been established for Hessian equations $F(D^2u)=0$ as a consequence of the (strong) Harnack inequality \cite{CC}. The case of more general uniformly elliptic inequalities $F(x,D^2u)+u^p\leq 0$ has been first and extensively developed in \cite{CLeoni} and, later, in \cite{FQ}. Different methods have been used to extend Liouville results to arbitrary Isaacs operators and unbounded domains in \cite{AS}.\\
 The case of PDEs with gradient dependent terms has been first addressed in \cite{CDC2} via a nonlinear version of the  Hadamard three-sphere theorem, see also \cite{CF,Rossi} for other contributions.
Some of the main contributions of this paper, that borrows several ideas from \cite{CLeoni,CDC2}, are degenerate extensions of the Hadamard-three sphere theorem for solutions to viscosity inequalities like \eqref{1} on the Heisenberg group. Here, crucial tools are the recent strong maximum and minimum principles for \eqref{1} developed in \cite{BG1} via the concept of generalized subunit vector fields.\\ We recall that for the Laplace equation in the plane, roughly speaking, the Hadamard three-circle theorem amounts to say that in a domain $\Omega\subset\R^2$ containing two concentric circles of radii $0<r_1<r_2$ and the region between them, if $\Delta u=u_{x_1x_1}+u_{x_2x_2}\leq0$ ($\geq0$) in $\Omega$, then for $r\in(r_1,r_2)$ the minimum (maximum) of $u$ over any concentric circle of radius $r$ is a concave (convex) function of $\log r$. Important byproducts of the three-circle property are Liouville-type theorems of the form (a) and (b) under suitable assumptions on the drift and zero-order terms, see \cite[Theorem 29 p.130]{PW} for the corresponding result for sub- and superharmonic functions in the plane. The Liouville properties we present here for PDEs of the form \eqref{1} seem to be the first in the literature within the context of fully nonlinear degenerate problems on Carnot groups, apart from those that already appeared in \cite{CTchou} for the mere Pucci-Heisenberg extremal operators that we recall below, and \cite{FV,BGL} for truncated Laplacian operators in the Euclidean setting, that however present different kind of degeneracies compared to our model PDEs. More general results for PDEs over H\"ormander vector fields will appear in \cite{BG2,BG3}.\\
Our model equations of the form \eqref{1} are those that we call, in analogy to \cite{CC}, \textit{uniformly subelliptic}, cf. \eqref{unifsubell} (see also \cite[Definition 1.1]{F}), whose prototype is
\[
\mathcal{M}^\pm_{\lambda,\Lambda}(D^2_{\mathbb{H}^d}u)+H(x,u,D_{\mathbb{H}^d}u)=0\text{ in }\mathbb{H}^d\ ,
\]
where $\mathcal{M}^\pm_{\lambda,\Lambda}(M)$ are the Pucci's operators with ellipticity constants $0<\lambda\leq\Lambda$ \cite{CC}
\[
\Mp(M)=\sup_{\lambda I_{2d}\leq A\leq \Lambda I_{2d}}\mathrm{Tr}(AM)=\Lambda\sum_{e_k>0}e_k+\lambda\sum_{e_k<0}e_k
\]
and
\[
\Mm(M)=\inf_{\lambda I_{2d}\leq A\leq \Lambda I_{2d}}\mathrm{Tr}(AM)=\lambda\sum_{e_k>0}e_k+\Lambda\sum_{e_k<0}e_k\ ,
\]
 $e_k=e_k(M)$ being the eigenvalues of $M\in\mathrm{Sym}_{2d}$, $D_{\mathbb{H}^d}u$ stands for the horizontal gradient and $D_{\mathbb{H}^d}^2u=(X_iX_ju)^*$ for the symmetrized horizontal Hessian of the unknown function $u$ in the Heisenberg group.\\

First, we complete and extend the results initiated in \cite{CLeoni,CTchou,BCDC} to fully nonlinear subelliptic equations in the Heisenberg group perturbed by semilinear terms. More precisely, we use a nonlinear degenerate Hadamard theorem \cite{CLeoni,CT,CTchou} to derive Liouville properties for viscosity solutions to model nonlinear PDEs of the form
\begin{equation}\label{semi}
\Mm(D^2_{\mathbb{H}^d} u)+h(x)u^p=0\text{ in }\mathbb{H}^d\ ,
\end{equation}
proving a property of type (b), i.e. that if $u$ is a nonnegative viscosity supersolution to \eqref{semi}, $h(x)\geq H|D_{\mathbb{H}^d}\rho|^2\rho^{\gamma}(x)$ for some $\gamma>-2$ and large $\rho$, $\rho$ being the homogeneous norm on the Heisenberg group, $H>0$ and $1<p\leq\frac{\beta+\gamma}{\beta-2}$, $\beta=\Lambda(\sQ-1)/\lambda+1$ and $\sQ=2d+2$ is the homogeneous dimension of the Heisenberg group, then $u$ must be identically zero. This can be seen as a nonlinear degenerate analogue to the seminal results in \cite{G} valid in the Euclidean framework and \cite[Theorem 1.1]{BCDC}, the latter proved for classical solutions to semilinear inequalities of the form $\Delta_{\mathbb{H}^d}u+H|D_{\mathbb{H}^d}\rho|^2\rho^\gamma u^p\leq0$ in $\mathbb{H}^d$ when $1<p<\frac{\sQ+\gamma}{\sQ-2}$ via different techniques than ours (note that, in fact, when $\lambda=\Lambda=1$ in \eqref{semi} we obtain $\beta=\sQ$, $\mathcal{M}_{\lambda,\lambda}(D^2_{\mathbb{H}^d} u)=\Delta_{\mathbb{H}^d}u$ and the previous critical threshold already found in \cite{BCDC}). The results we obtain here for the model PDE \eqref{semi} are new and sharp, see Counterexample \ref{sharpH}. Liouville properties in the linear case for solutions to subelliptic equations can be instead deduced via suitable Harnack inequalities, see \cite[Section 5.7 and 5.8]{BLU}, \cite[Theorem 1.1]{Tralli} and the references therein.\\

Finally, we extend the results obtained in \cite{CDC2} (see also \cite{CF}) for nonnegative solutions to the model problem
\begin{equation}\label{quasi}
\Mm(D^2_{\mathbb{H}^d} u)+\eta(\rho)|D_{\mathbb{H}^d}\rho||D_{\mathbb{H}^d}u|\leq0\text{ in }\mathbb{H}^d
\end{equation}
for a suitable radial function $\eta=\eta(\rho)$ fulfilling appropriate ``smallness" conditions, see \eqref{sigmal} and \eqref{sigma2} for precise assumptions. Under these restrictions, as in the Euclidean case, we are able to prove a degenerate version of the Hadamard theorem for \eqref{quasi} on suitable gauge balls that reads as follows: if $u$ is a viscosity solution to \eqref{quasi} in the annulus $A:=\{x\in\R^{2d+1}:r_1<\rho(x)<r_2\}$, then $m(r)=\min_{r_1\leq\rho(x)\leq r} u(x)$ is a concave function of
\[
 \psi(r)=-\int_{r_1}^r s^{-\frac{\Lambda}{\lambda}(\sQ-1)}\mathrm{exp}\left(\frac1\lambda\int_{r_1}^s\eta(\tau)d\tau\right)\,ds\ .
\]
(cf. \cite[eq. (3.5)]{CF}). The main difference with respect to the semilinear case $\eta=0$, where $ \psi$ has, in general, finite limit as $r\to\infty$ (cf. Theorem \ref{hadH} below), is that here $\psi$ is allowed to diverge due to the presence of the exponential term involving $\eta$ (see e.g. \cite[Theorem 4.3 and Corollary 4.1]{CF} for analogous results in the Euclidean setting). Thus, when $\psi$ diverges, we deduce Liouville type results of type (a) and (b). Under some assumptions on the drift terms, we establish even results for inequalities like
\[
\Delta_{\mathbb{H}^d} u+\eta(\rho)|D_{\mathbb{H}^d}\rho||D_{\mathbb{H}^d}u|\leq0\text{ in }\mathbb{H}^d\ ,
\]
which, to the best of our knowledge, seem to be new. 
We refer also to \cite{PucciH,BMaglia} for some Liouville results in the context of nonlinear PDEs in the Heisenberg group perturbed by (horizontal) gradient terms. Related results of Liouville-type for nonlinear viscosity inequalities in the Heisenberg group can be found in \cite{CDC,BCDC,BC,DLucente,MMT}, see also \cite{Ada} for Hadamard theorems for $p$-Laplacian type PDEs on Carnot groups. The complementary results for more general fully nonlinear degenerate PDEs of the form \eqref{1} perturbed by gradient terms can be found in \cite{BG2,BG3}.\\

\textit{Plan of the paper}. Section \ref{sec;not} contains some preliminary properties on the Heisenberg group and Section \ref{sec;max} comprises a review on maximum principles for fully nonlinear equations on the Heisenberg group. Section \ref{sec;semi} is devoted to the proof of the Liouville property for fully nonlinear equations perturbed by zero-order terms, while Section \ref{sec;grad} provides a new Hadamard-type result to handle PDEs with gradient terms.
\section{Preliminary notions}\label{sec;not}
The Heisenberg group $\He^d$ can be identified with $(\R^{2d+1},\circ)$, where $2d+1$ stands for the topological dimension and the (non-commutative) group multiplication $\circ$ is defined by
\begin{equation*}
x\circ y=\left(x_1+y_1,...,x_{2d}+y_{2d},x_{2d+1}+y_{2d+1}+2\sum_{i=1}^d(x_iy_{i+d}-x_{i+d}y_i)\right)\ .
\end{equation*}
The $d$-dimensional Heisenberg algebra is the Lie algebra spanned by the $m=2d$ left-invariant vector fields 
\begin{equation*}
X_i=\partial_i+2x_{i+d}\partial_{2d+1}\ ,
\end{equation*}
\begin{equation*}
X_{i+d}=\partial_{i+d}-2x_{i}\partial_{2d+1}\ ,
\end{equation*}
for $i=1,...,d$, $x$ standing for a generic point of $\R^{2d+1}$, $\partial_i=\partial_{x_i}$. Such vector fields satisfy the commutation relations
\[
[X_i,X_{i+d}]=-4\partial_{2d+1}\text{ and }[X_i,X_j]=0\text{ for all $j\neq i+d$, }i\in\{1,...,d\}
\]
and are 1-homogeneous with respect to the family of (anisotropic) dilations
\[
\delta_\eta(x)=(\eta x_1,...,\eta x_{2d},\eta^2x_{2d+1}) \ ,\eta>0\ .
\]
We denote by $\sQ=2d+2$ the homogeneous dimension of the Heisenberg group. Following \cite[Definition 5.1.1]{BLU}, it is useful to consider the following homogeneous norm defined via the stratification property of $\He^d$
\begin{equation}\label{homoH}
\rho(x)=\left(\left(\sum_{i=1}^{2d}(x_i)^2\right)^2+x_{2d+1}^2\right)^{\frac14}=(|x_H|^4+x_{2d+1}^2)^\frac14\ ,
\end{equation}
where $|x_H|$ refers to the Euclidean norm of the horizontal layer, which is 1-homogeneous with respect to the previous group of dilations.  We emphasize that from the computational viewpoint, the (gauge) norm \eqref{homoH} is easier to handle compared to the Carnot-Carath\'eodory norm. We finally recall that all homogeneous norms are equivalent in the context of Carnot groups (cf. \cite[Proposition 5.1.4]{BLU}).\\
We end the section with the following algebraic result that allows to write the horizontal Hessian (and its explicit eigenvalues) of a radial function on the Heisenberg group. 
\begin{lemma}\label{hesH}
Let $\rho$ be defined as in \eqref{homoH}. Then
\[
D^2_{\mathbb{H}^d}\rho=-\frac{3}{\rho}D_{\mathbb{H}^d}\rho\otimes D_{\mathbb{H}^d}\rho+\frac{1}{\rho}|D_{\mathbb{H}^d}\rho|^2I_{2d}+\frac{2}{\rho^3}\begin{pmatrix}B & C\\-C & B\end{pmatrix}\ ,
\]
where the matrices $B=(b_{ij})$ and $C=(c_{ij})$ are given by
\[
b_{ij}=x_ix_j+x_{d+i}x_{d+j}\ ,c_{ij}=x_ix_{d+j}-x_jx_{d+i}
\]
for $i,j=1,...,d$ (in particular $B=B^T$ and $C=-C^T$) and 
\[
|D_{\mathbb{H}^d}\rho|^2=\frac{|x_H|^2}{\rho^2}\leq 1\ .
\]
for $\rho>0$. In particular, for a radial function $f=f(\rho)$ we have
\[
D^2_{\mathbb{H}^d}f(\rho)=\frac{f'(\rho)|D_{\mathbb{H}^d}\rho|^2}{\rho}I_{2d}+2\frac{f'(\rho)}{\rho^3}\begin{pmatrix}B & C\\-C & B\end{pmatrix}+\left(f''(\rho)-3\frac{f'(\rho)}{\rho}\right)D_{\mathbb{H}^d}\rho\otimes D_{\mathbb{H}^d}\rho\ .
\]
Finally, the eigenvalues of $D^2_{\mathbb{H}^d}f(\rho)$ are $f''(\rho)|D_{\mathbb{H}^d}\rho|^2$, $3|D_{\mathbb{H}^d}\rho|^2f'(\rho)/\rho$ which are simple, and $|D_{\mathbb{H}^d}\rho|^2f'(\rho)/\rho$ which has multiplicity $2d-2$.
\end{lemma}
\begin{proof}
See \cite[Lemma 3.1 and Lemma 3.2]{CTchou}.
\end{proof}
Using this result, one obtains the following fundamental solutions to Pucci-Heisenberg operators, cf. \cite[Lemma 3.3]{CTchou}.
\begin{lemma}
Let $\alpha=\frac{\lambda}{\Lambda}(\sQ-1)\geq1$ and $\beta=\frac{\Lambda}{\lambda}(\sQ-1)\geq4$. The radial functions defined by
\[
\Phi_1(x)=f_1(\rho)\text{ and }\Phi_2(x)=f_2(\rho)
\]
with
\[
f_1(\rho)=\begin{cases}
C_1\rho^{2-\alpha}+C_2&\text{ if }\alpha<2;\\
C_1\log\rho+C_2&\text{ if }\alpha=2;\\
-C_1\rho^{2-\alpha}+C_2&\text{ if }\alpha>2;\
\end{cases}
\]
and
\[
f_2(\rho)=C_1\rho^{2-\beta}+C_2
\]
with $C_1\geq0$ and $C_2\in\R$, are classical solutions (and hence viscosity solutions) to the Pucci-Heisenberg minimal equation
\[
\Mm(D^2_{\mathbb{H}^d}\Phi)=0\text{ in }\R^{2d+1}\backslash\{0\}\ .
\]
In particular, $f_1$ is concave and increasing on $(0,\infty)$, while $f_2$ is convex and decreasing on $(0,\infty)$. Similarly,
\[
\Psi_1(x)=-\Phi_2(x)\text{ and }\Psi_2(x)=-\Phi_1(x)
\]
are classical solutions (and hence viscosity solutions) to the Pucci-Heisenberg maximal equation
\[
\Mp(D^2_{\mathbb{H}^d}\Psi)=0\text{ in }\R^{2d+1}\backslash\{0\}\ .
\]
\end{lemma}

\section{Maximum principles for fully nonlinear PDEs on the Heisenberg group}\label{sec;max}
Let $\Omega$ be an open set in $\mathbb{H}^d\simeq\R^{2d+1}$. In the theory of subelliptic equations of Heisenberg type one considers the (degenerate) PDE
\begin{equation*}
G(x,u,D_{\mathbb{H}^d}u,D_{\mathbb{H}^d}^2u)=0\text{ in }\Omega\ ,
\end{equation*}
where $D_{\mathbb{H}^d}u=(X_1u,...,X_{2d}u)$ is the horizontal gradient and $D_{\mathbb{H}^d}^2u$ is the symmetrized horizontal Hessian over the Heisenberg vector fields. In this setting, we say that a continuous real-valued function $G=G(x,t,p,M)$ on $\Omega\times\R\times\R^{2d}\times\mathrm{Sym}_{2d}$ is uniformly subelliptic if
\begin{equation}\label{unifsubell}
\mathcal{M}^-_{\lambda,\Lambda}(X)\leq G(x,t,p,X+Y)-G(x,t,p,Y)\leq \mathcal{M}^+_{\lambda,\Lambda}(X)
\end{equation}
for every $x\in\Omega,t\in\R,p\in\R^{2d}$ and $X,Y\in\mathrm{Sym}_{2d}$. 
 
After choosing a basis in the Euclidean space, if we take the matrix $\sigma$ whose columns are the coefficients of the vector fields, i.e. $X_j=\sigma^j\cdot D$, with $\sigma^j:\R^d\rightarrow\R^d$, and $\sigma=\sigma(x)=[\sigma^1(x),...,\sigma^m(x)]\in\R^{({2d+1})\times 2d}$, then
\begin{equation*}
D_{\mathbb{H}^d}u=\sigma^T Du=(\sigma^1\cdot Du,...,\sigma^m\cdot Du)
\end{equation*}
and
\begin{equation*}
D^2_{\mathbb{H}^d}u=\sigma^T D^2 u\ \sigma\ .
\end{equation*}
For general vector fields, the (symmetrized) horizontal Hessian presents an additional first order correction term (see \cite{BG1}). These considerations emphasize that the operator $G$ (hence $F$) is degenerate elliptic and hence, the natural notion of solution to handle such operators is that of viscosity solution, that we provide below for reader's convenience.
\begin{defn}\label{defvisc}
A lower (upper) semicontinuous function $u:\Omega\subset\R^d\to\R$, is a viscosity supersolution (subsolution) to $G(x,u,D_{\mathbb{H}^d}u,(D_{\mathbb{H}^d}^2u)^*)=0$ if for all $\varphi\in C^2(\Omega)$ and $x_0\in\Omega$ such that $u-\varphi$ has a local minimum (maximum) at $x_0\in\Omega$, we have
\[
G(x_0,u(x_0),D_{\mathbb{H}^d}\varphi(x_0),D_{\mathbb{H}^d}^2\varphi(x_0))\leq0\ (\geq 0)\ .
\]
If $u$ is a viscosity supersolution (subsolution) we will also say that $u$ verifies 
\[
G(x,u,D_{\mathbb{H}^d} u,D_{\mathbb{H}^d}^2u)\leq 0\ (\geq 0)
\]
in the viscosity sense.
\end{defn}
Following the above discussion, see also \cite{BG1}, exploiting the representation via Euclidean coordinates, i.e.
\[
F(x,r,p,X)=G(x,r,\sigma^T(x)p,\sigma^T(x)X\sigma(x)+g(x,Du))\ ,
\]
one finds an equivalent degenerate equation in the Euclidean setting with $F$ having $\sigma^i$ as generalized subunit vector fields (cf. \cite[Lemma 3.1]{BG1}). 
Therefore, the results of \cite{BG1} ensure the validity of the strong maximum and minimum principle for fully nonlinear Pucci-Heisenberg operators, see \cite[eq. (15)]{BG1} and \cite[Corollary 3.2]{BG1}, as stated by the next theorem
\begin{thm}
Let $\eta=\eta(\rho)$ be a continuous real-valued function and $u$ be a viscosity supersolution to
\[
\Mm(D_{\mathbb{H}^d}^2u)+\eta(\rho)|D_{\mathbb{H}^d}\rho||D_{\mathbb{H}^d}u|=0\text{ in }\Omega\ .
\]
If $u$ attains its minimum at $x_0\in\Omega$, then $u$ is constant. Similarly, if $u$ is a viscosity subsolution to
\[
\Mp(D_{\mathbb{H}^d}^2u)+\eta(\rho)|D_{\mathbb{H}^d}\rho||D_{\mathbb{H}^d}u|=0\text{ in }\Omega\ .
\]
and $u$ attains its maximum at $x_0\in\Omega$, then $u$ is constant.
\end{thm}
As for the comparison principle, we recall the following results proved in \cite{BM}, which is rewritten here in a suited form for our purposes.
\begin{thm}\label{comp1}
Let $\Omega$ be a bounded open set and $u$ and $v$ be, respectively, a viscosity supersolution and subsolution to $\Mm(D_{\mathbb{H}^d}^2u)=f(x)$ or $\Mp(D_{\mathbb{H}^d}^2u)=f(x)$, $f\in C(\Omega)$, such that $u\geq v$ on $\partial \Omega$. Then $u\geq v$ on $\overline{\Omega}$.
\end{thm}
\begin{thm}\label{comp2}
Let $G[u]=G_1(D^2_{\mathbb{H}^d}u)+H_1(x,D_{\mathbb{H}^d}u)$ with $G_1$ uniformly subelliptic and $H_1$ satisfying
\[
|H_1(x,\sigma^T(x)p)|\leq C_1|\sigma^T(x)p|+C_2\ ,
\]
\[
|H_1(x,\sigma^T(x)p)-H_1(y,\sigma^T(y)p)|\leq\omega(|x-y|)(1+|p|)\ ,
\]
for some modulus of continuity $\omega$. Let $u$ and $v$ be, respectively, a viscosity supersolution and subsolution to $G[u]=0$ in the annulus $A:=\{x\in\R^{2d+1}:r_1<\rho(x)<r_2\}$ satisfying  $u\geq v$ on $\partial A$. Then $u\geq v$ on $\overline{A}$.
\end{thm}

\section{PDEs perturbed by zero-order terms}\label{sec;semi}
\subsection{Hadamard theorems and some consequences}
We start by recalling the following nonlinear degenerate Hadamard-type result for sub- and supersolutions to Pucci-Heisenberg equations, see \cite[Theorem 5.1]{CTchou} for the proof.
\begin{thm}\label{hadH}
Let $\Omega$ be a domain in $\R^{2d+1}$ containing the closed intrinsic ball $\overline{B}_{\mathbb{H}^d}(0,R):=\{x\in\R^{2d+1}:\rho(x)\leq r_2\}$, $r_2>0$. Then
\begin{itemize}
\item[(i)] If $u\in \mathrm{LSC}(\Omega)$ is a viscosity solution of
\[
\Mp(D^2_{\mathbb{H}^d}u)\leq0\text{ in }\Omega\ ,
\]
then the function $m(r)=\min_{\rho(x)\leq r}u(x)$, $r<r_2$, is respectively a concave function of $\log(r)$ if $\alpha=2$ and of $r^{2-\alpha}$ if $\alpha\neq 2$, with $\alpha=\lambda(\sQ-1)/\Lambda+1\geq1$. More precisely, we have
\begin{equation}\label{mG}
m(r)\geq \frac{G(r)}{G(r_2)}m(r_2)+\left(1-\frac{G(r)}{G(r_2)}\right)m(r_1)
\end{equation}
for all $r_1\leq r\leq r_2$ with
\[
G(r)=\begin{cases}
\log(r/r_1)&\text{ if }\alpha=2\ ,\\
r^{2-\alpha}-r_1^{2-\alpha}&\text{ if }\alpha\neq2\ .
\end{cases}
\]
\item[(ii)] If $u\in \mathrm{LSC}(\Omega)$ is a viscosity solution of
\[
\Mm(D^2_{\mathbb{H}^d}u)\leq0\text{ in }\Omega,
\]
then the function $m(r)=\min_{\rho(x)\leq r}u(x)$ is a concave function of $r^{2-\beta}$, with $\beta=\Lambda(\sQ-1)/\lambda+1\geq4$. More precisely, we have
\[
m(r)\geq \frac{G(r)}{G(r_2)}m(r_2)+\left(1-\frac{G(r)}{G(r_2)}\right)m(r_1)
\]
for all $r_1\leq r\leq r_2$ with
\[
G(r)=r^{2-\beta}-r_1^{2-\beta}\ .
\]
\end{itemize}
\end{thm}
\begin{rem}\label{incr}
It is straightforward to check that if $u$ is a viscosity solution to
\[
u\geq0\ ,\ \Mm(D^2_{\mathbb{H}^d}u)\leq 0\text{ in }\mathbb{H}^d
\]
(respectively to
\[
u\geq0\ ,\ \Mp(D^2_{\mathbb{H}^d}u)\leq 0\text{ in }\mathbb{H}^d)\ ,
\]
then the function $r\in[0,+\infty)\longmapsto m(r)r^{\beta-2}$($r\in[0,+\infty)\longmapsto m(r)r^{\alpha-2}$) is increasing. Indeed, if $u$ is nonnegative and $\Mm(D^2_{\mathbb{H}^d}u)\leq 0$, using Theorem \ref{hadH}-(ii), letting $r_2\to\infty$ and being $m(r_2)\geq0$, one concludes
\[
m(r)\geq m(r_1)\frac{r_1^{\beta-2}}{r^{\beta-2}}
\] 
for all $r\geq r_1$. In the second case one argues similarly noting that for $r\geq r_1$ one has
\[
m(r)\geq\begin{cases}
m(r_1)&\text{ if }\alpha=2\\
m(r_1)\frac{r_1^{\alpha-2}}{r^{\alpha-2}}&\text{ if }\alpha\neq2\ .
\end{cases}
\]
\end{rem}
By the duality relation among Pucci's extremal operators, one can formulate Theorem \ref{hadH} in terms of the function $M(r)=\max_{\rho(x)\leq r}u(x)$, see \cite{CLeoni,CTchou}.
A consequence of the previous result for Pucci-Heisenberg operators involves normalized $p$-Laplacian equations \cite{MPR,Pina} defined as
\begin{equation}\label{np}
\frac1p|D_{\mathbb{H}^d} u|^{2-p}\mathrm{div}_{\mathbb{H}^d}(|D_{\mathbb{H}^d} u|^{p-2}D_{\mathbb{H}^d} u)=0\text{ in }\mathbb{H}^d\ ,1<p<\infty\ ,
\end{equation}
where $\mathrm{div}_{\mathbb{H}^d}(\Phi(x))=\sum_{i=1}^{2d}X_i\Phi(x)$, $\Phi:\R^{2d+1}\to\R$, is the horizontal divergence over the Heisenberg vector fields. It is straightforward to check that the operator \[\Delta_{p,\mathbb{H}^d}^N u=\frac1p|D_{\mathbb{H}^d} u|^{2-p}\mathrm{div}_{\mathbb{H}^d}(|D_{\mathbb{H}^d} u|^{p-2}D_{\mathbb{H}^d} u)\] can be equivalently rewritten as
\[
\Delta_{p,\mathbb{H}^d}^N u=\mathrm{Tr}[A(D_{\mathbb{H}^d} u)(D^2_{\mathbb{H}^d} u)^*]
\]
with 
\[
A(D_{\mathbb{H}^d} u)=\frac1pI_d+\left(1-\frac{2}{p}\right)\frac{D_{\mathbb{H}^d} u\otimes D_{\mathbb{H}^d} u}{|D_{\mathbb{H}^d} u|^2}\\ ,
\]
or, in other words, $\Delta_{p,\mathbb{H}^d}^N u=(1/p)\Delta_{\mathbb{H}^d} u+(1-2/p)\Delta_{\mathbb{H}^d,\infty}^1 u$, where $\Delta_{\mathbb{H}^d,\infty}^h u=|D_{\mathbb{H}^d} u|^{h-3}\Delta_{\mathbb{H}^d,\infty}u$ is the $h$-homogeneous $\infty$-Laplacian equation studied in \cite{BieskeMartin}. It is immediate to see that
\[
\min\left\{\frac1p,\frac{p-1}{p}\right\}|\xi|^2\leq A(D_{\mathbb{H}^d} u)\xi\cdot \xi\leq\max\left\{\frac1p,\frac{p-1}{p}\right\}|\xi|^2\ ,
\]
showing that $\Delta_{p,\mathbb{H}^d}^N$ is uniformly subelliptic for $p\in(1,\infty)$.
Therefore, such nonlinear operator satisfies the inequalities
\[
\mathcal{M}^-_{\lambda,\Lambda}(D^2_{\mathbb{H}^d} u)\leq \Delta_{p,\mathbb{H}^d}^N u\leq \mathcal{M}^+_{\lambda,\Lambda}(D^2_{\mathbb{H}^d} u)
\]
 for $\lambda=\min\{\frac1p,\frac{p-1}{p}\}$ and $\Lambda=\max\{\frac1p,\frac{p-1}{p}\}$, see e.g. \cite[Lemma 4.1]{Kawohl} and references therein.
\begin{cor}
Let $\Omega$ be a domain in $\R^{2d+1}$ containing the closed intrinsic ball $\overline{B}_{\mathbb{H}^d}(0,r_2):=\{x\in\R^{2d+1}:\rho(x)\leq r_2\}$, $r_2>0$, and consider three gauge spheres of radii $0<r_1<r<r_2$. Then, if $u\in C(\Omega)$ is a viscosity solution to
\[
\Delta_{p,\mathbb{H}^d}^Nu\leq0\text{ in }\Omega\ , p>1\ ,
\]
then the function $m(r)=\min_{\rho(x)\leq r}u(x)$ is a concave function of $r^{2-\beta}$, with $\beta=\Lambda(\sQ-1)/\lambda+1$. More precisely, we have
\[
m(r)\geq \frac{G(r)}{G(r_2)}m(r_2)+\left(1-\frac{G(r)}{G(r_2)}\right)m(r_1)
\]
for all $r_1\leq r\leq r_2$ with
\[
G(r)=r^{2-\beta}-r_1^{2-\beta}\ .
\]
\end{cor}
\begin{proof}
It suffices to observe that $u\in C(\Omega)$ solves
\[
0\geq \Delta_{p,\mathbb{H}^d}^Nu\geq \Mm(D^2_{\mathbb{H}^d}u)\text{ in }\mathbb{H}^d
\]
for $\lambda=\min\{\frac1p,\frac{p-1}{p}\}$ and $\Lambda=\max\{\frac1p,\frac{p-1}{p}\}$ in the viscosity sense, and apply Theorem \ref{hadH}-(ii).
\end{proof}
As outlined in \cite[Theorem 5.3]{CTchou}, the above Hadamard theorem for Pucci-Heisenberg equations yields a weak Harnack-type inequality for radial solutions to fully nonlinear degenerate PDEs over Heisenberg vector fields, that we prove for reader's convenience owing to the strong maximum principle \cite{BG1}. We remark that, to our knowledge, this is the only Harnack-type inequality available for fully nonlinear degenerate PDEs over H\"ormander vector fields. 

\begin{cor}
Let $u\in \LSC(\R^{2d+1})$ be a radial viscosity solution to
\[
G(x, D^2_{\mathbb{H}^d}u)\leq0\text{ in }B_{\mathbb{H}^d}(0,2R)
\]
with $u\geq0$ and $G$ satisfying \eqref{unifsubell}. Set  also $ \beta=\frac{\Lambda}{\lambda}(\sQ-1)+1$. Then, $u$ satisfies the following weak Harnack inequality:
\[
\mathrm{meas}(B_{\mathbb{H}^d}(0,R/2)\cap \{u>t\})\leq \frac{C}{t^{\frac{\sQ}{ \beta-2}}}\left(\inf_{B_{\mathbb{H}^d}(0,R)}u\right)^{\frac{\sQ}{ \beta-2}}\ \forall t>0
\]
where $\mathrm{meas}(E)$ denotes the measure of the measurable set $E\subset \R^{2d+1}$.
\end{cor}

\begin{proof}
We have
\[
u\geq0\text{ and }\Mm(D^2_{\mathbb{H}^d}u)\leq0
\]
in the viscosity sense. The strong minimum principle \cite{BG1} yields that $\min_{x\in \overline{B}_{\mathbb{H}^d}(0,R)}u$ is attained at the boundary of the Koranyi ball $\overline{B}_{\mathbb{H}^d}(0,R)$, i.e. at points where $\{\rho=R\}$. In view of Remark \ref{incr} one also observes that the function
\[
r\longmapsto m(r)r^{\beta-2}
\]
is increasing, cf. \cite[Corollary 3.1]{CLeoni}. Being also $m(\rho)=u(\rho)$, we deduce
\[
\mathrm{meas}(B_{\mathbb{H}^d}(0,R/2)\cap \{m(\rho)>t\})\leq \mathrm{meas}\left(B_{\mathbb{H}^d}(0,R/2)\cap \left\{\frac{m(R)R^{\beta-2}}{\rho^{\beta-2}}>t\right\}\right)\leq \frac{CR^\beta}{t^{\frac{\sQ}{\beta-2}}}(m(R))^{\frac{\sQ}{\beta-2}}\ .
\]

\end{proof}
\subsection{Liouville theorems}
We recall for reader's convenience that the one-side Liouville theorem for subharmonic functions on the Heisenberg group (i.e. solving $\Delta_{\mathbb{H}^d}u\geq0$) fails, see e.g. \cite{BG2} for explicit counterexamples. However, the Liouville property can be recovered for the Pucci-Heisenberg operators over horizontal Hessians as a consequence of Theorem \ref{hadH} and the strong maximum and minimum principle in \cite{BG1}. The next result has already appeared in \cite[Theorem 5.2]{CTchou}, and we propose the proof here for self-containedness, being a consequence of Theorem \ref{hadH}. However, it does not allow to infer Liouville properties for general fully nonlinear operators on $\mathbb{H}^d$ as in \cite{BG2,BC} via property \eqref{unifsubell}.
\begin{thm}\label{lioP}
Let $u\in C(\R^{2d+1})$ be a viscosity solution either bounded from below to $\Mp(D^2_{\mathbb{H}^d}u)\leq 0$ in $\R^{2d+1}$ or bounded from above to $\Mm(D^2_{\mathbb{H}^d}u)\geq0$ in $\R^{2d+1}$, and $\alpha\leq 2$ (i.e. $\sQ\leq\Lambda/\lambda+1$), then $u$ is constant. 
\end{thm}
\begin{proof}
Consider the case $\Mp(D^2_{\mathbb{H}^d}u)\leq 0$. By Theorem \ref{hadH}-(i) we have that $m$ satisfies \eqref{mG} for every fixed $0<r_1<r_2$. Since $m(r)$ is a bounded function, due to the fact that $u$ is bounded from below, and being $\alpha\leq2$, one concludes, after letting $r_2\to\infty$,
\[
m(r)\geq m(r_1)\text{ for }r\geq r_1>0\ .
\]
Since $r\longmapsto m(r)$ is a decreasing function, we immediately deduce that $m(r)$ is constant and, in particular, $m(r)\equiv m(0)=u(0)$. Therefore, $u$ attains its minimum at an interior point and, by the strong minimum principle \cite[Example 3.8]{BG1}, $u$ is constant.
\end{proof}
An important byproduct of Theorem \ref{hadH} is the following Liouville theorem for fully nonlinear second order PDEs with zero order terms on the Heisenberg group, which is the main new result of this section.
\begin{thm}\label{lioH}
Let $ \beta=\frac{\Lambda}{\lambda}(\sQ-1)+1(\geq 4)$ and $u\in C(\R^{2d+1})$ be a nonnegative viscosity solution to
\[
G(x,D^2_{\mathbb{H}^d}u)+f(x,u(x))\leq0\text{ in }\R^{2d+1}\ ,
\]
where  $G$ satisfies \eqref{unifsubell} and $f$ is a nonnegative function satisfying 
\[
f(x,u(x))\geq h(x)u^p(x)
\]
for some $h\in C(\R^{2d+1})$ nonnegative such that for $\rho$ large, say $\rho\geq \rho_0>0$, we have
\begin{equation}\label{h}
h(x)\geq H|D_{\mathbb{H}^d}\rho|^2\rho^\gamma(x)\ ,
\end{equation}
where $H>0$ and $\gamma>-2$. Then $u\equiv 0$ provided that $1<p\leq\frac{ \beta+\gamma}{ \beta-2}$.
\end{thm}

\begin{proof}
We follow the strategy implemented in \cite{CLeoni}, adapting it to the subelliptic framework of the Heisenberg group. Since $h\geq0$ and $G$ is uniformly subelliptic via \eqref{unifsubell}, by comparison with Pucci-Heisenberg operators we have
\[
u\geq0\ ,\ \mathcal{M}^-_{\lambda,\Lambda}(D^2_{\mathbb{H}^d}u)\leq0
\]
 in the viscosity sense. We assume that $u>0$ in $\R^{2d+1}$. In fact, by the strong minimum principle, we have either $u\equiv 0$ or $u>0$ in $\R^{2d+1}$. Indeed, if there exists a point $x_0$ where $u$ attains its minimum $u(x_0)=0$, by the strong minimum principle \cite{BG1} applied on the ball $B_{\mathbb{H}^d}(x_0,R)$, $R>0$, one concludes that $u\equiv 0$ throughout this ball. Then, by the arbitrariness of $R>0$, one obtains $u\equiv 0$ on the whole space. Therefore, we will assume that $u>0$ in $\R^{2d+1}$.\\
 
 For every $r$, consider $m(r)=\min_{\rho(x)\leq r}u(x)$ and observe that by the assumptions $m$ is strictly positive, and by the strong minimum principle it is strictly decreasing with respect to $r$. In addition, in view of Remark \ref{incr} we have 
\begin{equation}\label{b1}
m(r)\leq \frac{m(R)R^{\beta-2}}{r^{\beta-2}}\ ,R\geq r>0\ .
\end{equation}
Consider the radial $C^2$ function
\[
\varphi(x)=m(r)\left\{1-\frac{[(\rho(x)-r)^+]^3}{(R-r)^3}\right\}\ ,
\]
where $R>r\geq r_0$ are arbitrarily fixed. We observe that $\varphi(x)\leq0<u(x)$ for $\rho(x)\geq R$ and $\varphi(x)\equiv m(r)<u(x)$ for $\rho(x)<r$. Therefore, being $\varphi(x)=u(x)$ at least at one point where $\rho(x)=r$, the minimum of $u-\varphi$ in $\R^{2d+1}$ is nonpositive and achieved at some point $x_R^r\in \R^{2d+1}$ in the annulus $r\leq\rho(x_R^r)<R$. We then apply the definition of viscosity solution (cf. Definition \ref{defvisc}) using $\varphi$ as a test function at $x_R^r$ to deduce
\[
\Mm(D^2_{\mathbb{H}^d}\varphi(x_R^r))+h(x_R^r)(u(x_R^r))^p\leq0\ .
\]
To compute the minimal Pucci-Heisenberg operator over $D^2_{\mathbb{H}^d}\varphi(x_R^r)$, we observe that, by Lemma \ref{hesH}, for every $x\in\R^{2d+1}\backslash\{0\}$ the eigenvalues of $D^2_{\mathbb{H}^d}\varphi$ are
\[
-6|D_{\mathbb{H}^d}\rho|^2\frac{m(r)}{(R-r)^3}(\rho(x)-r)^+\ ,
\]
\[
-9|D_{\mathbb{H}^d}\rho|^2\frac{m(r)}{\rho(x)(R-r)^3}[(\rho(x)-r)^+]^2\ ,
\]
which are both simple, and
\[
-3|D_{\mathbb{H}^d}\rho|^2\frac{m(r)}{\rho(x)(R-r)^3}[(\rho(x)-r)^+]^2\ ,
\]
which instead has multiplicity $2d-2$, giving
\[
\Mm(D^2_{\mathbb{H}^d}\varphi(x))=-\frac{3\Lambda m(r)}{(R-r)^3}(\rho(x)-r)^+|D_{\mathbb{H}^d}\rho|^2\left[2+\frac{\sQ-1}{\rho(x)}(\rho(x)-r)^+\right]
\]
being $\sQ=2d+2$. Using \eqref{h} and the fact that $\rho(x_R^r)\geq r>r_0$ we further deduce
\[
(u(x_R^r))^p\leq \frac{3\Lambda m(r)}{H(R-r)^3}\frac{(\rho(x_R^r)-r)^+}{\rho(x_R^r)^{\gamma}}\left[2+\frac{\sQ-1}{\rho(x_R^r)}(\rho(x_R^r)-r)^+\right]\ .
\]
If $\rho(x_R^r)=r$ we get the contradiction $u(x_R^r)=0$. Therefore, we assume $r<\rho(x_R^r)<R$ and obtain
\[
(u(x_R^r))^p\leq \frac{3\Lambda m(r)(\sQ+1)}{H(R-r)^2\rho^{\gamma}(x_R^r)}\ .
\]
Owing to this latter fact and using the inequality $u(x_R^r)\geq m(R)$ we get
\[
m(R)\leq C\frac{R^{\gamma^-/p}m(r)^{1/p}}{r^{\gamma^+/p}(R-r)^{2/p}}
\]
for all $R>r\geq r_0$, where $\gamma=\gamma^+-\gamma^-$, $\gamma^+,\gamma^->0$ and
\[
C=\left(\frac{3\Lambda (\sQ+1)}{H}\right)^\frac1p\ .
\]
Exploiting \eqref{b1} we deduce
\[
m(R)\leq C\frac{R^{(\gamma^-+\beta-2)/p}m(R)^{1/p}}{r^{(\gamma^++\beta-2)/p}(R-r)^{2/p}}
\]
from which, choosing $r=R/2$, it follows
\[
m(R)\leq \tilde C\frac{m(R)^{1/p}}{R^{(\gamma+2)/p}}
\]
for all $R>2r_0$, where $\tilde C$ is a possibly different constant from $C$. Consider first the case $1<p<\frac{\beta+\gamma}{\beta-2}$. Then
\[
m(R)\leq \frac{\tilde C^{\frac{p}{p-1}}}{R^{\frac{\gamma+2}{p-1}}}\ ,
\] 
which implies
\begin{equation}\label{upperH}
R^{\beta-2}m(R)\leq \frac{\tilde C^{\frac{p}{p-1}}}{R^{\frac{\gamma+2}{p-1}-\beta+2}}
\end{equation}
for all $R>2r_0$. Letting $R\to\infty$ one obtains $R^{\beta-2}m(R)$ would tend to zero, contradicting the fact that $R^{\beta-2}m(R)$ is a positive and increasing function by Remark \ref{incr}. \\
We finally consider the borderline exponent $p=\frac{\beta+\gamma}{\beta-2}$. In this case, from \eqref{upperH} one obtains an upper bound for the function $R^{\beta-2}m(R)$. We now reach a contradiction exploiting this latter information as follows. Fix $R_1\geq r_0$, $c_1>0$ and $c_2\in\R$. For any $x\in\R^{2d+1}$ such that $\rho(x)\geq R_1$, we consider the radial function
\[
z(x)=g(\rho(x))=c_1\frac{\log(1+\rho(x))}{\rho^{\beta-2}(x)}+c_2
\]
and we choose
\[
R_1\geq \mathrm{exp}\left(\frac{2\beta-3}{(\beta-2)(\beta-1)}\right)-1>\mathrm{exp}\left(\frac{1}{\beta-2}\right)-1
\]
so that $z$ is a convex and non-increasing function of the homogeneous norm $\rho$. Then, one notices that the eigenvalues of the horizontal Hessian $D^2_{\mathbb{H}^d}z$ are, cf. Lemma \ref{hesH},
\[
c_1|D_{\mathbb{H}^d}\rho|^2\left[(2-\beta)(1-\beta)\rho^{-\beta}\log(1+\rho)+2(2-\beta)\frac{\rho^{1-\beta}}{1+\rho}-\frac{\rho^{2-\beta}}{(1+\rho)^2}\right]\ ,
\]
\[
3c_1|D_{\mathbb{H}^d}\rho|^2\left[(2-\beta)\rho^{-\beta}\log(1+\rho)+\frac{\rho^{1-\beta}}{1+\rho}\right]\ ,
\]
which are simple, and
\[
c_1|D_{\mathbb{H}^d}\rho|^2\left[(2-\beta)\rho^{-\beta}\log(1+\rho)+\frac{\rho^{1-\beta}}{1+\rho}\right]
\]
with multiplicity $2d-2$. We then get, recalling that $\beta=\Lambda(\sQ-1)/\lambda+1$,
\begin{multline*}
\Mm(D^2_{\mathbb{H}^d}z(x))=c_1|D_{\mathbb{H}^d}\rho|^2\left\{\lambda\rho^{-\beta}(2-\beta)[1-\beta+\Lambda(\sQ-1)/\lambda]\log(1+\rho)+\lambda2(2-\beta)\frac{\rho^{1-\beta}}{1+\rho}\right.\\
\left.+\Lambda(\sQ-1)\frac{\rho^{1-\beta}}{1+\rho}-\lambda\frac{\rho^{2-\beta}}{(1+\rho)^2}\right\}\\
\geq -\lambda c_1|D_{\mathbb{H}^d}\rho|^2\left[\frac{\beta-3}{(1+\rho)\rho^{\beta-1}}+\frac{1}{(1+\rho)^2\rho^{\beta-2}}\right]\geq -\lambda c_1\frac{\beta-2}{\rho^{\beta}}|D_{\mathbb{H}^d}\rho|^2
\end{multline*}
for all $x\in\R^{2d+1}$ such that $\rho(x)\geq R_1$. We now arbitrarily choose $R_2$ such that $R_2>R_1$ and determine $c_1>0$ and $c_2\in\R$ such that
\[
z(R_1)\leq m(R_1)\text{ and }z(R_2)=m(R_2)\ ,
\]
from which one is allowed to pick $c_1$ such that
\[
0<c_1\leq \frac{m(R_1)-m(R_2)}{\frac{\log(1+R_1)}{R_1^{\beta-2}}-\frac{\log(1+R_2)}{R_2^{\beta-2}}}\ ,
\]
which is possible since $R_1$ is chosen so that the radial function $z$ is decreasing. This forces $c_2=m(R_2)-c_1\frac{\log(1+R_2)}{R_2^{\beta-2}}$. We therefore observe that
\[
z\leq u\text{ on }\{\rho(x)=R_1\}\cup \{\rho(x)=R_2\}
\]
and
\[
\Mm(D^2_{\mathbb{H}^d}z(x))\geq -\lambda c_1\frac{\beta-2}{\rho^{\beta}}|D_{\mathbb{H}^d}\rho|^2\text{ in }\{R_1<\rho(x)<R_2\}\ .
\]
Then, recalling that
\[
u(x)\geq m(\rho(x))\geq \frac{m(R_1)R_1^{\beta-2}}{\rho^{\beta-2}}
\]
and using \eqref{h}, for $R_1\geq r_0$ we get
\[
\Mm(D^2_{\mathbb{H}^d}u(x))\leq -\frac{H(m(R_1)R_1^{\beta-2})^{\frac{\beta+\gamma}{\beta-2}}}{\rho^\beta}|D_{\mathbb{H}^d}\rho|^2\text{ in }\{R_1<\rho(x)<R_2\}
\]
in the viscosity sense. Possibly choosing a smaller constant $c_1>0$ one obtains
\[
\Mm(D^2_{\mathbb{H}^d}u(x))\leq -\frac{C|D_{\mathbb{H}^d}\rho|^2}{\rho^{\beta}}\leq \Mm(D^2_{\mathbb{H}^d}z(x))\ .
\]
By the comparison principle in Theorem \ref{comp1} (see \cite{BM,M,BG1}, all of them applied with $H(x,D_{\mathbb{H}^d} u)=0$) we get
\[
z\leq u \text{ in }\{R_1\leq \rho(x)\leq R_2\}
\]
for all $R_2>R_1$. We then let $R_2\to\infty$, noting that $c_2\to0$, we observe that
\[
u(x)\geq \tilde{c}\frac{\log(1+\rho)}{\rho^{\beta-2}}
\]
for some $\tilde c>0$ and for all $x\in\R^{2d+1}$ such that $\rho\geq R_1$. In particular, for all $R\geq R_1$ we conclude
\[
R^{\beta-2}m(R)\geq \tilde{c}\log(1+R)
\]
which contradicts \eqref{upperH}.
\end{proof}

\begin{rem}
The results of Theorem \ref{lioH} generalize to the fully nonlinear setting \cite[Theorem 3.1]{BCDC}. This is easily seen by taking $\lambda=\Lambda$ so that $\Mm(D^2_{\mathbb{H}^d}u)\equiv\Delta_{\mathbb{H}^d} u$ and $\beta=\sQ$, whence $p\in \left(1,\frac{\sQ+\gamma}{\sQ-2}\right]$. Our Theorem \ref{lioH} extends to the subelliptic setting \cite[Theorem 4.1]{CLeoni}. Let us stress out that the condition on $p$ in the Euclidean case reads as $1<p\leq\frac{\frac{\Lambda}{\lambda}(d-1)+1+\gamma}{\frac{\Lambda}{\lambda}(d-1)-1}$. Here, the ``subelliptic" dimension $\sQ=2d+2$ replaces the dimension of the ambient space $\R^d$ of the Euclidean case, as expected.
\end{rem}
We also point out that the upper bound on the exponent $p$ in Theorem \ref{lioH} is sharp, as the next counterexamples show
\begin{cex}[from \cite{BCDC}] First, set $\lambda=\Lambda=1$, so that the Pucci's extremal operator coincides with the Heisenberg sub-Laplacian. The function $v(\rho)=C(1+\rho)^{\frac{\epsilon+2-\sQ}{2}}$ is a positive solution to 
\[
\Delta_{\mathbb{H}^d} u(x)+\rho^\gamma(x)|D_{\mathbb{H}^d}\rho|^2 u^p\leq0\text{ in }\R^{2d+1}
\]
for $p\geq \frac{\sQ+\gamma-\epsilon}{\sQ-2-\epsilon}$ for every $\epsilon>0$ and a suitable constant $C>0$.
\end{cex} 

\begin{cex}\label{sharpH}
Theorem \ref{lioH} gives the sharp result. Indeed, there exists a non-trivial classical solution when $p>\frac{\beta+\gamma}{\beta-2}$ and $\gamma>-2$ to the degenerate elliptic inequality
\[
u\geq0\ ,\mathcal{M}^-_{\lambda,\Lambda}(D^2_{\mathbb{H}^d}u)+|D_{\mathbb{H}^d}\rho|^2(1+\rho^2)^{\frac{\gamma}{2}}u^p\leq0\text{ in }\mathbb{H}^d\ .
\]
Consider the function
\[
u_\delta(x)=C_\delta(1+\rho^2)^{-\delta}\ ,x\in\R^{2d+1}
\]
with $\delta,C_\delta>0$ to be determined. We observe that for all $x$ such that $\rho^2(x)\leq 1/(2\delta+1)$, $u_\delta$ is a concave and decreasing radial function of $\rho$. Using that $\Lambda(\sQ-1)=\lambda(\beta-1)$ we have
\begin{multline*}
\Mm(D^2_{\mathbb{H}^d}u_\delta(x))=-2\Lambda C_\delta \delta|D_{\mathbb{H}^d}\rho|^2\left[\frac{\sQ-1}{(1+\rho^2)^{\delta+1}}-\frac{(2\delta+1)\rho^2-1}{(1+\rho^2)^{\delta+2}}\right]\\
\leq -\frac{2\Lambda(\sQ-1)C_\delta\delta}{(1+\rho^2)^{\delta+1}}|D_{\mathbb{H}^d}\rho|^2=-\frac{2\lambda(\beta-1)C_\delta\delta}{(1+\rho^2)^{\delta+1}}|D_{\mathbb{H}^d}\rho|^2\ .
\end{multline*}
If, instead, $\rho^2(x)\geq 1/(2\delta+1)$, then $u_\delta$ is a decreasing, convex function of $\rho$ and satisfies
 \begin{multline*}
 \Mm(D^2_{\mathbb{H}^d}u_\delta(x))=-2\lambda C_\delta \delta|D_{\mathbb{H}^d}\rho|^2\left[\frac{\beta-1}{(1+\rho^2)^{\delta+1}}-\frac{(2\delta+1)\rho^2-1}{(1+\rho^2)^{\delta+2}}\right]\\
 =-2\lambda C_\delta\delta|D_{\mathbb{H}^d}\rho|^2\frac{[\beta-2(\delta+1)]\rho^2+\beta}{(1+\rho^2)^{\delta+2}}\leq -2\lambda C_\delta\delta|D_{\mathbb{H}^d}\rho|^2\frac{\beta-2(\delta+1)}{(1+\rho^2)^{\delta+1}}\ .
 \end{multline*}
 In both cases analysed above we conclude
\[
   \Mm(D^2_{\mathbb{H}^d}u_\delta(x))+|D_{\mathbb{H}^d}\rho|^2(1+\rho^2)^{\frac{\gamma}{2}}u_\delta^p\leq \left(-2\lambda C_\delta\delta\frac{\beta-2(\delta+1)}{(1+\rho^2)^{\delta+1}}+\frac{C_\delta^p}{(1+\rho^2)^{p\delta-\gamma/2}}\right)|D_{\mathbb{H}^d}\rho|^2\ .
\]
If the constants are chosen so that
\[
2(\delta+1)<\beta\ ,
\]
\[
\delta+1\leq p\delta-\gamma/2\ ,
\]
\[
C_\delta^{p-1}\leq 2\lambda\delta[\beta-2(\delta+1)]\ ,
\]
we conclude that the right-hand side of the above inequality is nonpositive. It is sufficient to take $\delta=\frac{\beta-2-\epsilon}{2}$ for every $\epsilon>0$ and such that $\epsilon<\beta-2$, whereas $p\geq\frac{\beta+\gamma-\epsilon}{\beta-2-\epsilon}$ and $C_\delta^{p-1}\leq \lambda\epsilon(\beta-2-\epsilon)$.
\end{cex}
\begin{rem}
We conclude saying that imposing additional assumptions on $G$ in Theorem \ref{lioH} one can extend the range of the exponent $p$ guaranteeing the Liouville property, cf. \cite[Remark 8]{CLeoni}. For instance, let $G=\Mp$ and consider the nonlinear problem
\[
u\geq 0\ , \Mp(D^2_{\mathbb{H}^d}u)+h(x)u^p\leq 0\text{ in }\mathbb{H}^d\ .
\]
Then, $u$ satisfies
\[
u\geq 0\ , \lambda\Delta_{\mathbb{H}^d}u+h(x)u^p\leq 0\text{ in }\mathbb{H}^d\ .
\]
in the viscosity sense. Using a similar approach to Theorem \ref{lioH} (or the results in \cite[Theorem 1.1]{BCDC}), one can conclude that $u\equiv0$ provided
\[
1<p\leq \frac{\sQ+\gamma}{\sQ-2}\ ,\gamma>-2\ .
\]
However, by Remark \ref{incr} the radial function $r\longmapsto m(r)r^{\alpha-2}$, $\alpha=\lambda(\sQ-1)/\Lambda+1$ is increasing. Therefore, using similar arguments to those in Theorem \ref{lioH} by replacing $\beta$ with $\alpha$, one concludes that $u$ vanishes whenever
\[
1<p\leq \frac{\alpha+\gamma}{\alpha-2}\ ,\gamma>-2\ .
\]
Note that $\alpha<\sQ$ since $\lambda<\Lambda$, $\gamma>-2$, and hence $\frac{\alpha+\gamma}{\alpha-2}> \frac{\sQ+\gamma}{\sQ-2}$. 
\end{rem}

\section{PDEs perturbed by gradient terms}\label{sec;grad}
\subsection{Hadamard-type results}
In this section we prove some nonlinear degenerate version of the Hadamard three-sphere theorem that appeared in \cite{CDC2} in the Euclidean case. We will make the following assumptions
\begin{equation}\label{0}
G(x,0,0,0)=0\ ,
\end{equation}
\begin{equation}\label{sopra}
G(x,t,p,0)\geq \eta(\rho)|D_{\mathbb{H}^d}\rho||p|+h(x)t^\alpha
\end{equation}
$\alpha\geq1$, $\sigma$ and $h$ are continuous and satisfy
\begin{equation}\label{sigmal}
\rho\eta(\rho)\leq\Lambda(\sQ-1)
\end{equation}
and $h(x)\geq0$. We point out that the assumption \eqref{sopra} on the first order term is natural in this setting, see \cite[Section 4]{CTchou}. We prove the following
\begin{thm}\label{had1}
Let $\Omega$ be a domain of $\R^{2d+1}$ containing the closed intrinsic ball $B_{\mathbb{H}^d}(0,r_2)$, $r_2>0$. Let $u$ be solving
\begin{equation}
u\geq0\ ,G(x,u,D_{\mathbb{H}^d}u,D_{\mathbb{H}^d}^2u)\leq0\text{ in }\Omega
\end{equation}
in the viscosity sense. Suppose that $G$ is uniformly subelliptic and \eqref{0},\eqref{sopra} and \eqref{sigmal} hold. Then
\[
m(r)=\min_{r_1\leq\rho\leq r}u\ ,r<r_2
\]
is a concave function of
\[
\psi(r)=-\int_{r_1}^r s^{-\frac{\Lambda}{\lambda}(\sQ-1)}\mathrm{exp}\left(\frac1\lambda\int_{r_1}^s\eta(\tau)d\tau\right)\,ds\ ,
\]
namely we have
\begin{equation}\label{m}
m(r)\geq \frac{\psi(r)}{\psi(r_2)}m(r_2)+\left(1-\frac{\psi(r)}{\psi(r_2)}\right)m(r_1)
\end{equation}
for all $r_1\leq r\leq r_2$.
\end{thm}
\begin{rem}
Note that for $\eta=0$ we recover Theorem \ref{hadH}-(ii).
\end{rem}
\begin{proof}
We first observe that under the standing assumptions $u$ is a nonnegative viscosity solution to 
\[
\Mm(D^2_{\mathbb{H}^d}u)+\eta(\rho)|D_{\mathbb{H}^d}\rho| |D_{\mathbb{H}^d}u|\leq0\text{ in }A=\{x\in\R^{2d+1}:r_1<\rho(x)<r_2\}\ .
\]
We then consider the Dirichlet problem satisfied by a smooth radial function $f=f(\rho)$
\[
\Mm(D^2_{\mathbb{H}^d}f(\rho))+\eta(\rho)|D_{\mathbb{H}^d}\rho||D_{\mathbb{H}^d}f(\rho)|=0
\]
equipped with the boundary conditions $f(r_1)=m(r_1)$ and $f(\rho_2)=m(r_2)$. Using Lemma \ref{hesH} we note that the eigenvalues of $D^2_{\mathbb{H}^d}f(\rho)$ are $|D_{\mathbb{H}^d}\rho|^2f''(\rho)$ and $3|D_{\mathbb{H}^d}\rho|^2 f'(\rho)/\rho$ which are simple, and $|D_{\mathbb{H}^d}\rho|^2f'(\rho)/\rho$ which has multiplicity $2d-2$. Therefore, any radial, convex and non-increasing solution of the Dirichlet problem solves
\[
\lambda|D_{\mathbb{H}^d}\rho|^2\left[ f''(\rho)+\left(\frac{\Lambda}{\lambda}\frac{(\sQ-1)}{\rho}-\frac{\eta(\rho)}{\lambda}\right)f'(\rho)\right]=0
\]
equipped with the boundary conditions $f(r_1)=m(r_1)$ and $f(\rho_2)=m(r_2)$. This gives immediately by integrating the above ODE
\[
f(r)=\frac{\psi(r)}{\psi(r_2)}m(r_2)+\left(1-\frac{\psi(r)}{\psi(r_2)}\right)m(r_1)\ .
\]
By assumptions $f(\rho(x))\leq u(x)$ on $\partial A$, thus by the comparison principle Theorem \ref{comp2} we obtain
\begin{equation}\label{ineqann}
u(x)\geq f(\rho(x))\text{ in }\overline{A}\ .
\end{equation} 
We now observe that the claim is trivial if $u$ is a constant. Assume that $u$ is not a constant. By the strong minimum principle, $u$ must attain its minimum on the boundary of $A$. If $m(r)=\min_{\rho=r_1}u$, then the definition of $m$, the boundary condition on $f$ and the fact that $f$ is non-increasing give
\[
m(r)=m(r_1)=f(r_1)\geq f(r)
\]
If, instead, $m(r)=\max_{\rho=\tilde r}u(x)=u(\tilde x_r)$ for some $\rho(\tilde x_r)=\tilde r$, then from \eqref{ineqann} we deduce $m(r)\geq f(\rho(\tilde x_r))=f(\tilde r)$. Therefore, we conclude $m(r)\geq f(r)$ for all $r\in[r_1,r_2]$.
\end{proof}
\begin{rem}
Observe that $\psi$ is non-increasing and 
\[
 \psi''(r)=\mathrm{exp}\left(\frac1\lambda\int_{r_1}^r\eta(\tau)d\tau\right)\left(\frac{\Lambda}{\lambda}(\sQ-1)r^{-\frac{\lambda}{\Lambda}(\sQ-1)-1}-\frac{\eta(r)}{\lambda}r^{-\frac{\Lambda}{\lambda}(\sQ-1)}\right)\ .
\]
Therefore, by assumption \eqref{sigmal} we have that $\psi$ is convex.
\end{rem}
An analogous result holds for subsolutions under ``reversed assumptions" on the operator $G$. More precisely, we assume
\begin{equation}\label{upper}
G(x,t,p,0)\leq \eta(\rho)|D_{\mathbb{H}^d}\rho||p|+h(x)t^\alpha
\end{equation}
where $\alpha\geq1$ and $\sigma,h$ are continuous real-valued functions  such that
\begin{equation}\label{sigma2}
\rho\eta(\rho)\geq- \lambda(\sQ-1)
\end{equation}
\begin{equation}\label{h2}
h(x)\leq0
\end{equation}
and prove the following
\begin{thm}
Let $\Omega$ be a domain of $\R^{2d+1}$ containing the closed intrinsic ball $B_{\mathbb{H}^d}(0,r_2)$, $r_2>0$. Let $u$ be solving
\begin{equation}
u\leq0\ ,G(x,u,D_{\mathbb{H}^d}u,D_{\mathbb{H}^d}^2u)\geq0\text{ in }\Omega
\end{equation}
in the viscosity sense. If $G$ is uniformly subelliptic and satisfies \eqref{0},\eqref{upper}, \eqref{sigma2} and \eqref{h2}, then
\[
M(r)=\max_{r_1\leq\rho\leq r}u\ ,r<r_2
\]
is a convex function of
\[
\tilde \psi(r)=\int_{r_1}^r s^{-\frac{\lambda}{\Lambda}(\sQ-1)}\mathrm{exp}\left(-\frac1\Lambda\int_{r_1}^s\eta(\tau)d\tau\right)\,ds\ ,
\]
namely we have
\begin{equation}\label{m}
M(r)\leq \frac{\tilde \psi(r)}{\tilde \psi(r_2)}M(r_2)+\left(1-\frac{\tilde \psi(r)}{\tilde \psi(r_2)}\right)M(r_1)
\end{equation}
for all $r_1\leq r\leq r_2$.\end{thm}

\subsection{Liouville properties}
We are now ready to prove Liouville-type theorems via the Hadamard properties of the previous section. 
\begin{thm}\label{lio1}
Let $u\geq0$ be a viscosity solution to
\[
G(x,u,D_{\mathbb{H}^d}u,D_{\mathbb{H}^d}^2u)\leq0\text{ in }\R^{2d+1}\ .
\]
with $G$ uniformly subelliptic and satisfying \eqref{0},\eqref{sopra} and \eqref{sigmal}. If
\begin{equation}\label{psiinf}
\lim_{r\to\infty}\psi(r)=-\infty\ ,
\end{equation}
then $u$ is a constant. Moreover, if $h>0$ at some point $x_0\in\R^{2d+1}$, then $u\equiv0$.
\end{thm}
\begin{proof}
Since $u$ is a viscosity solution to $G[u]\leq0$ on $\R^{2d+1}$, it is a supersolution on the annulus $A$. By Theorem \ref{had1}, since $m(r_2)\geq0$, one obtains
\[
m(r)\geq m(r_1)\left(1-\frac{\psi(r)}{\psi(r_2)}\right)\text{ for every }r\in[r_1,r_2]
\]
If we keep $r$ fixed and let $r_2\to\infty$, using \eqref{psiinf} we conclude
\[
m(r)\geq m(r_1)\text{ for }r\geq r_1\ .
\]
Since $m$ is non-increasing, we conclude that $m(r)=m(r_1)$ for every $r\geq r_1$.  Hence, $m(r)\equiv m(0)=u(0)$, that is $u$ attains its minimum at the interior point $x=0$. Since $u$ is a viscosity solution to
\[
\Mm(D_{\mathbb{H}^d}^2u)+\eta(\rho)|D_{\mathbb{H}^d}\rho||D_{\mathbb{H}^d}u|\leq0
\]
in any annulus, by the strong minimum principle $u$ is a constant and the claim is proved. Now observe that if $u(x)=C$ is a nonnegative constant solution to $G[u]=0$, then
\[
h(x_0)C^\alpha\leq G(x_0,C,0,0)\leq0\ ,
\]
which implies that $u\equiv0$ if the strict inequality $h>0$ holds at some point $x_0\in\R^{2d+1}$.
\end{proof}
\begin{rem}
We remark that the condition on the behaviour at infinity of $\psi$ is necessary for the validity of the Liouville property. In fact, if $\lim_{r\to\infty}\psi(r)=-L$, $L\in(0,\infty)$, then the positive function $L+\psi$ solves
\[
\Mm(D_{\mathbb{H}^d}^2u)+\eta(\rho)|D_{\mathbb{H}^d}\rho||D_{\mathbb{H}^d}u|=0\text{ in }\R^{2d+1}\backslash\{0\}
\]
in classical sense, while the constant function $L$ solves the same equation on the whole space. Thus
\[
v(x)=\begin{cases}
L+\psi(\rho(x))&\text{ if }\rho(x)\geq r_1\\
L&\text{ if }\rho(x)<r_1\ ,
\end{cases}
\]
is a strictly positive non-trivial continuous viscosity solution. Classical stability properties then ensure that $u\equiv \min\{v,L\}$ is a viscosity supersolution of the equation on the whole space.
\end{rem}

\subsection{Remarks on the assumption \eqref{psiinf}
}
We now briefly discuss assumption \eqref{psiinf} in relation with the behaviour of the radial function $\eta$. 

\begin{itemize}
\item We observe that if
\begin{equation}\label{finiteeta}
\int_{r_1}^{+\infty}|\eta(\tau)|d\tau=M<\infty,
\end{equation}
then
\[
\psi(r)\leq \mathrm{exp}(M/\lambda)\int_{r_1}^rs^{-\frac{\Lambda}{\lambda}(\sQ-1)}
\]
and thus $\psi$ is convergent as $r\to\infty$ for every $\lambda,\Lambda$ since $\sQ\geq4$, so Theorem \ref{lio1} cannot be applied if $\eta$ satisfies \eqref{finiteeta}. 
\item When \begin{equation}\label{etainfty}\int_{r_1}^{+\infty}|\eta(\tau)|d\tau=\infty\end{equation} one can apply the Liouville results stated in Theorem \ref{lio1}. A class of $\eta$'s satisfying the assumption \eqref{psiinf} of Theorem \ref{lio1} is
\[
\frac{\Lambda(\sQ-1)-\lambda}{\rho}\leq \eta(\rho)\leq   \frac{\Lambda(\sQ-1)}{\rho}
\]
when $\rho$ is large enough, so that $\psi(r)\simeq-\log r$, for large $r$. In particular, when $\lambda=\Lambda=1$, $\mathcal{M}^-$ reduces to $\Delta_{\mathbb{H}^d}$. Therefore, under \eqref{etainfty} Theorem \ref{lio1} contains as special case a Liouville property for linear inequalities of the form
\[
\Delta_{\mathbb{H}^d}u+\eta(\rho)|D_{\mathbb{H}^d}\rho|||D_{\mathbb{H}^d}u| \leq0\  (\geq0)\text{ in }\mathbb{H}^d\ .
\]

\end{itemize}

\end{document}